\documentclass[a4paper, 11pt]{article}

\usepackage[T1]{fontenc}
\usepackage{lmodern}
\usepackage{textcomp}
\usepackage{microtype} % Modern addition: vastly improves spacing and typography

\usepackage[a4paper, margin=1in]{geometry}
\usepackage[singlespacing]{setspace}
\usepackage[displaymath]{lineno}
\usepackage{enumitem}
\usepackage{calc}
\usepackage{hhline}
\usepackage[dvipsnames]{xcolor}

\usepackage{amsmath, amsfonts, amssymb, amsthm}
\usepackage{mathtools} % Extends and improves amsmath
\usepackage{mathrsfs}
\usepackage{upgreek}
\usepackage{extarrows}
\usepackage{bbm}
\usepackage{cmupint}
\usepackage{bm}
\usepackage{graphicx}
\usepackage{tikz}
\usepackage{tikz-cd}
\usepackage{amscd}
\usepackage[all,cmtip]{xy}

\usepackage{cite} 
\usepackage{hyperref}

\hypersetup{
	colorlinks=true,
	breaklinks=true,
	urlcolor=blue,
	citecolor=BrickRed,
	linkcolor=teal,
	bookmarksopen=false,
	pdftitle={A Frobenius Theorem on Fr\'echet Manifolds},
	pdfauthor={Kaveh Eftekharinasab}
}

\usepackage{bookmark}
\makeatletter
\let\orgdescriptionlabel\descriptionlabel
\renewcommand*{\descriptionlabel}[1]{%
	\let\orglabel\label
	\let\label\@gobble
	\phantomsection
	\edef\@currentlabel{#1}%
	\let\label\orglabel
	\orgdescriptionlabel{#1}%
}
\makeatother

\theoremstyle{plain}
\newtheorem{theorem}{Theorem}[section]
\newtheorem{lemma}[theorem]{Lemma}
\newtheorem{corollary}[theorem]{Corollary}

\newtheorem{proposition}[theorem]{Proposition}
\newtheorem{remark}[theorem]{Remark}

\theoremstyle{definition}
\newtheorem{definition}[theorem]{Definition}
\newtheorem*{definition*}{Definition}

\numberwithin{equation}{section}

\DeclareMathOperator{\Lip}{Lip_{loc}}

\newcommand{\id}{\mathrm{Id}}

\newcommand{\abs}[1]{\lvert#1\rvert}

\newcommand{\snorm}[2][]{\left\lVert#2\right\rVert_{#1}}

\newcommand{\TB}[1]{\mathsf{#1}}
\newcommand{\Vfields}[3]{\eu{V}_{C_c}^{#1}(#2, #3)}

\newcommand{\va}{\varphi}
\newcommand{\fs}[1]{\mathsf {#1}}
\newcommand{\eu}[1]{\EuScript {#1}}

\newcommand{\vf}[1]{\mathcal{#1}}
\DeclareMathAlphabet{\mathpzc}{OT1}{pzc}{m}{it}
\DeclareMathAlphabet{\mathsfit}{OT1}{cmss}{m}{sl}
\DeclareMathAlphabet\EuScript{U}{eus}{m}{n}
\SetMathAlphabet\EuScript{bold}{U}{eus}{b}{n}

\renewcommand{\emptyset}{\varnothing}
\newcommand{\set}[1]{\left\{#1\right\}}
\newcommand\Set[2]{\left\{#1\mid#2\right\}} 
\newcommand{\zero}[1]{\boldsymbol{0}_{#1}}

\newcommand{\pc}{\partial_{\mathrm{c}}}

\newcommand{\bl}[1]{\textup{#1}}

\newcommand{\fr}{Fr\'{e}chet}

\providecommand{\U}[1]{\protect\rule{.1in}{.1in}}

\newcommand{\subjclass}[1]{\textbf{AMS Subject Classifications (2020):} #1\par}
\newcommand{\keywords}[1]{\textbf{Keywords:} #1\par}

\def\raisefix#1{%
	\ifx#1\displaystyle\raise.39ex\else
	\ifx#1\textstyle\raise.39ex\else
	\ifx#1\scriptstyle\raise.275ex\else
	\raise.150ex\fi\fi\fi
}
\def\stylefix#1{%
	\ifx#1\displaystyle\scriptstyle\else
	\ifx#1\textstyle\scriptstyle\else
	\ifx#1\scriptstyle\scriptscriptstyle\else
	\scriptscriptstyle\fi\fi\fi
}

\sbox0{$\fontdimen8\textfont3=0.4pt$}
\DeclareFontFamily{U}{mathx}{\hyphenchar\font45}
\DeclareFontShape{U}{mathx}{m}{n}{%
	<5> <6> <7> <8> <9> <10>
	<10.95> <12> <14.4> <17.28> <20.74> <24.88>
	mathx10
}{}

\title{A Frobenius Theorem on Fr\'echet Manifolds}
\author{Kaveh Eftekharinasab\\[0.5ex]
	\small Algebra and Topology Department, Institute of Mathematics of the \\
	\small National Academy of Sciences of Ukraine \\
	\small Tereshchenkivska st.~3, 01024, Kyiv, Ukraine \\ 
	\small \texttt{kaveh@imath.kiev.ua}}
\date{}

\begin{document}
	
	\maketitle
	
\begin{abstract}
	We investigate the integrability of Fr\'echet tangent distributions on Fr\'echet manifolds. We introduce the local well-posedness Condition~W for split tangent subbundles, which reduces the local integrability problem to solving initial value problems with parameters whose solutions define curves tangent to the distribution. By applying a variational approach to establish the existence and uniqueness of these solutions, we prove a Frobenius theorem stating that involutivity and Condition~W are sufficient for integrability. This yields the existence of a unique maximal foliation of the manifold. Furthermore, we provide a dual formulation of the theorem using differential forms, which characterizes the algebraic conditions for integrability via the exterior derivative of the subbundle's local annihilator.
\end{abstract}


% \linenumbers % Uncomment if you want line numbers active

\section{Introduction}\label{sec:introduction}
We establish a version of the Frobenius Theorem and its formulation in terms of foliations for infinite-dimensional Fr\'echet tangent distributions on Fr\'echet manifolds. While the Frobenius theorem is well-established for Banach manifolds, its extension to the Fr\'echet framework is obstructed by the failure of the Picard--Lindel\"of theorem in non-normable locally convex spaces. This failure implies that differentiable vector fields may fail to admit local flows, or that such flows may lack continuous dependence on initial conditions. Consequently, the split tangent subbundle cannot be locally integrated into a family of maximal connected integral submanifolds unless additional conditions are imposed to guarantee the existence of a local foliation. Thus, the involutivity of a tangent subbundle is a necessary but insufficient condition for integrability. 

Several approaches and frameworks have been developed to study the Frobenius theorem beyond the context of Banach manifolds. In~\cite{tich}, Teichmann established a version of the Frobenius theorem for finite-rank involutive subbundles on convenient manifolds by employing the local finite-dimensionality of the distribution. Later, Eyni~\cite{eyn} investigated Banach distributions on manifolds modeled on locally convex spaces. More recently, Pelletier~\cite{f} and Cabau and Pelletier~\cite{cabau} have focused on the integrability of distributions within the frameworks of projective and direct limits of Banach manifolds, respectively. In this paper, we address the integrability problem for general Fr\'echet distributions that do not necessarily admit a Banach structure or a representation as a limit of Banach spaces.

To this end, we introduce the local well-posedness Condition~W for split tangent subbundles (Definition~\ref{def:condition_W}). This condition reformulates the problem of local integrability as that of solving initial value problems (IVPs) with parameters, the solutions of which define tangent curves of the distribution. We treat these IVPs variationally by associating them with an appropriately constructed functional~\(\eu{I}\). By requiring that~\(\eu{I}\) satisfies Palais--Smale type conditions, we establish the existence and uniqueness of critical points, which correspond precisely to the solutions of the IVPs, thereby providing the analytical basis for integrating involutive distributions and constructing the leaves of the resulting foliation. We note, however, that the Palais--Smale conditions can be technically challenging to verify and may be difficult to establish in certain settings; for further discussion and concrete examples, we refer the reader to~\cite{k5}.

In Proposition~\ref{prop:efc}, we prove that a split tangent subbundle is locally integrable provided it is both involutive and satisfies Condition~W. By virtue of this result, in Theorem~\ref{thm:global_frobenius} we prove that these conditions are sufficient to establish the existence of a global atlas of Frobenius charts and that every point in the manifold belongs to a unique maximal connected integral submanifold. In Theorem~\ref{thm:global_frobenius_foliation}, we establish the one-to-one correspondence between the integrability of the subbundle and the existence of a foliation on the Fr\'echet manifold. Finally, Corollary~\ref{cor:frobenious} provides a dual formulation of integrability in terms of differential forms, characterized via the exterior derivative of the subbundle's local annihilator.

\section{Preliminaries}
In this section, we briefly summarize the required framework and notation, following~\cite{kau}. We assume that \( (\fs{F}, \textsf{Sem}({\fs{F}})) \) and \( (\fs{E}, \textsf{Sem}({\fs{E}})) \) are Fr\'echet spaces over \( \mathbb{R} \), where 
\(\textsf{Sem}({\fs{F}}) = (\snorm[\fs{F},n]{\cdot})_{n \in \mathbb{N}}\)
and \(\textsf{Sem}({\fs{E}}) = (\snorm[\fs{E},n]{\cdot})_{n \in \mathbb{N}}\)
are increasing sequences of  seminorms that define the topologies of \( \fs{F} \) and \( \fs{E} \), respectively. 
We denote by \( \zero{\fs{X}} \)  the origin  of a \fr\ space \( \fs{X} \).

Let \( \mathfrak{S} \) denote the family of all compact subsets of \( \fs{F} \). Let \( \mathcal{L}_c(\fs{F}, \fs{E}) \) be the space of all continuous linear mappings from \( \fs{F} \) to \( \fs{E} \), endowed with the topology of compact convergence. 
This topology is Hausdorff, locally convex, and is defined by the family of seminorms
\[
\left\lVert \ell \right\rVert_{S,i} \coloneqq \sup \set{ \snorm[\fs{E},i]{\ell(f)} \mid f \in S },
\]
where \( S \in \mathfrak{S} \) and \( i \in \mathbb{N} \). 
When \( \fs{E} = \mathbb{R} \) (with the standard absolute value \( \left\lvert \cdot \right\rvert \)), the dual space \( \fs{F}'_c = \mathcal{L}_c(\fs{F}, \mathbb{R}) \) is endowed with the topology defined by the family of seminorms \( \{ \left\lVert \cdot \right\rVert_{S} \mid S \in \mathfrak{S}\} \).

\begin{definition}[Definition 1.0.0, \cite{ke}]\label{def:diff}
	Let \(U \subseteq  \mathsf {E}\) be open, and  $ \varphi\colon U   \to  \mathsf {F}$  a mapping. Then the  derivative
	of $\varphi$ at $x$ in the direction $h$ is defined by 
	\[
	\mathrm{D}\varphi(x)(h) \coloneqq \lim_{t \to 0} \frac{1}{t} \left( \varphi(x+th) - \varphi(x) \right),
	\]
	provided the limit exists. 
	The mapping \(\varphi\) is called differentiable at \(x\) if \(\mathrm{D} \varphi(x)(h)\) exists for all \(h \in \fs{E}\). 
	Furthermore, \(\varphi\) is called a \(C_c^1\)-mapping if it is differentiable at all points of \(U\) and the derivative map 
	\(
	\mathrm{D} \varphi \colon U \to \mathcal{L}_c(\fs{E}, \fs{F})
	\) is continuous. 
\end{definition}
Higher order derivatives are defined naturally (Definition 2.5.0, \cite{ke}). For a continuous curve $ \gamma\colon I=(a,b) \to \fs{F}  $, we define its derivative by
\begin{equation*}
	\gamma'(t) = \lim_{h \to 0} \frac{\gamma(t+h)-\gamma(t)}{h}.
\end{equation*}
If the limit exists and is finite, and $ \gamma'(t) $ is continuous, we say that $ \gamma $ is $ C^1 $, a notion which coincides with  Keller's $C_c^1$-differentiability. If $ I=[a,b] $, the extension of the derivative by continuity of $ \gamma' $ to $ [a,b] $ has the values $ \gamma'(a) $ and $ \gamma'(b) $ equal to
\begin{equation*}
	\gamma'(a) = \lim_{ h \downarrow 0} \dfrac{\gamma(a+h)-\gamma(a)}{h}, \quad \gamma'(b) = \lim_{ h \downarrow 0} \dfrac{\gamma(b)-\gamma(b-h)}{h}.
\end{equation*}
We denote by $ \fs{C}^k({I}, \fs{F})$ the \fr\ space of all Keller's $C_c^k$-mappings from ${I}$ to $\fs{F}$ for $k \geq 1$.
\begin{remark}
	Employing this class of mappings is primarily motivated by the requirement for a suitable topology on dual spaces to  define the Palais--Smale condition. These mappings are known as Keller's \(C_c^k\)-mappings. Notably, this notion of differentiability is equivalent to the well-established and widely used Michal--Bastiani notion.
\end{remark}

\begin{definition}[Definition 1.1, \cite{eftekharinasab}]\label{def:PS}
	Let $\va \colon \fs{F} \to \mathbb{R}$ be a Keller's $C_c^1$-functional.
	\begin{itemize}
		\item[(i)] We say that $\va$ satisfies the Palais--Smale condition (PS-condition) if every sequence $(x_i) \subset \fs{F}$ such that the sequence $(\va(x_i))$ is bounded and
		\(
		\mathrm{D} \va(x_i) \xrightarrow{\fs{F}_c'} 0
		\),	possesses a convergent subsequence.
		\item[(ii)] We say that $\va$ satisfies the Palais--Smale condition at level $m \in \mathbb{R}$ ($\mathrm{(PS)_m}$-condition) if every sequence $(x_i) \subset \fs{F}$ such that
		\(\va(x_i) \to m \) and \( \mathrm{D} \va(x_i) \xrightarrow{\fs{F}_c'} 0\), possesses a convergent subsequence.
	\end{itemize}
\end{definition}

To broaden the applicability of our results, we also consider functionals that are merely locally Lipschitz. Consequently, we require elements of non-smooth analysis to handle such mappings. In \cite{k1}, a critical point theory for these mappings, generalizing the Clarke subdifferential, was developed. Below, we review the concepts necessary for our subsequent results.

Let \( \langle \fs{F}, \fs{F}' \rangle \) be a dual pairing. 
%The weak topology \( \sigma(\fs{F},\fs{F}') \) on \( \fs{F} \) is defined by the family of seminorms:
%\begin{equation*}
%\snorm[A]{y} \coloneqq \sup_{x' \in A} \abs{ \langle x', y \rangle}, \quad \text{for } y \in \fs{F},
%\end{equation*}
%where \( A \) ranges over the finite subsets of \( \fs{F}' \). 
The weak* topology \( \sigma^*(\fs{F}',\fs{F}) \) on \( \fs{F}' \) is defined by the family of seminorms
\begin{equation*}\label{eq:wss}
	\snorm[B]{x} \coloneqq \sup_{y \in B} \abs{ \langle x, y \rangle}, \quad \text{for } x \in \fs{F}',
\end{equation*}
where \( B \) ranges over the set \(\Phi_{\fs{F} } \) of all finite subsets of \( \fs{F} \). Let \( \Lip(\fs{F}, \mathbb{R}) \) denote the set of locally Lipschitz functionals on \( \fs{F} \). Following the definitions in \cite{k1}, for \( \va \in \Lip(\fs{F}, \mathbb{R}) \), the generalized directional derivative \( \va^{\circ}(x,y) \) at \( x \in \fs{F} \) in the direction \( y \in \fs{F} \) is defined by
\begin{equation*}\label{def:gg}
	\va^{\circ}(x,y) \coloneqq \limsup_{h \to x, t \downarrow 0} \frac{\va(h+ty) - \va(h)}{t}.
\end{equation*}
The Clarke subdifferential of \( \va \) at \( x \) is the set-valued mapping \( \pc \va \colon \fs{F} \rightrightarrows \fs{F}' \) defined by
\begin{equation*}
	\pc \va(x) \coloneqq \left\{ x' \in \fs{F}' \mid \langle x', y \rangle \leq \va^{\circ}(x,y) \text{ for all } y \in \fs{F} \right\}.
\end{equation*}
The set \( \pc \va(x) \) is weak*-compact, thus
\begin{gather*}\label{eq:chpsc}
	\lambda_{\va,B} \colon \fs{F} \to \mathbb{R}, \quad
	\lambda_{\va,B}(x) = \min_{z \in \pc \va(x)} \snorm[B]{z},\, B \in \Phi_{\fs{F} } 
\end{gather*}
is a well-defined function.
\begin{definition}[Definition 2.1, \cite{k1}]\label{def:ChangPS}
	Let $\va \in \Lip(\fs{F}, \mathbb{R})$. 
	We say that $\va$ satisfies the Palais--Smale condition in Chang's sense, or the Chang PS-condition, if every sequence $(x_i) \subset \fs{F}$ such that the sequence $(\va(x_i))$ is bounded and 
	\begin{equation}\label{eq:cpsc}
		\lim_{i \to \infty} \lambda_{\va,B}(x_i) = 0 \quad \text{for each finite subset } B \subset \fs{F},
	\end{equation}
	possesses a convergent subsequence. Additionally, if every sequence $(x_i) \subset \fs{F}$ satisfying \eqref{eq:cpsc} and such that $\va(x_i) \to m \in \mathbb{R}$ possesses a convergent subsequence, we say that $\va$ satisfies the Chang PS-condition at level $m$ (denoted as $(\mathrm{PS})_m$).
\end{definition}
Now, we will define a class of functionals that we will need.  Let \(\fs{X}\) be a \fr\ space.
Let \(\bl{AU}(\fs{X},[0,\infty))\) be the set of all functionals \(\eu{I}\colon\fs{X}\to [0,\infty)\) such that either
\begin{enumerate}
	\item \(\eu{I}\) is a Keller's \(C^1_c\)-functional satisfying
	\(\eu{I}(x)=0 \iff x=\mathbf{0}_{\fs{X}}\)\label{itm:sgdp.1} and
	\(\mathrm{D}\eu{I}(y)=0 \iff y=\mathbf{0}_{\fs{X}}\).\label{itm:sgdp.2}
	\item \(\eu{I} \in \Lip(\fs{X}, \mathbb{R})\)  satisfying
	\(\eu{I}(x)=0 \iff x=\mathbf{0}_{\fs{X}}\)\label{itm:gdp1} and
	\(\mathbf{0}_{\fs{X}'}\in\pc\eu{I}(y) \iff y=\mathbf{0}_{\fs{X}}\)\label{itm:gdp2}.
\end{enumerate}
Consider the following initial value problem depending on a parameter $e$
\begin{equation}\label{eq:ivp}
	\begin{cases}
		\eu{Y}'(t) = \upphi(t, \eu{Y}(t), e), \quad \forall t \in \bl{I} \\
		\eu{Y}(t_0) = f.
	\end{cases}
\end{equation}
Here, the values $\eu{Y}(t)$ belong to the Fr\'echet space $\fs{F}$, $a > 0$, $t_0 \in \mathbb{R}$, and $e$ belongs to the Fr\'echet space $\fs{E}$. The interval is given by $\bl{I} = [t_0-a, t_0+a]$. Additionally, we suppose that $\upphi \colon \bl{I} \times \fs{F} \times \fs{E} \to \fs{F}$ is a Keller's $C_c^1$-mapping. 
\begin{theorem}[Theorem 4.1, \cite{kau}]\label{th:ivp}
	Let \( \eu{I} \in \bl{AU}(\mathbb{R} \times \fs{F} \times \fs{E}, [0,\infty)) \), and suppose that the following condition holds$\colon$
	\begin{enumerate}[label={{\textup{(C)}}},ref=C]
		\item\label{itm:ivp.1} For any \( \psi \in \fs{C}^2 ([-1,1],\fs{F}) \) and each fixed \( s \in [-1,1] \), the mapping
		\begin{gather*}
			\eu{J}_{\psi,s}\colon \mathbb{R} \times \fs{F} \times \fs{E} \to [0,\infty) \\
			\eu{J}_{\psi,s} (a,f,e) = \eu{I}\big(\psi'(s) - a\upphi(t_0+as, \psi(s)+f,e)\big)
		\end{gather*}
		satisfies the $\mathrm{PS}$-condition at any level if \( \eu{I} \) is a Keller's \( C^1_c \)-functional. Moreover, it satisfies the Chang \textup{PS}-condition at any level if \( \eu{I} \) is locally Lipschitz. 
	\end{enumerate}
	Then, there exists \( b \in (0,a] \) such that the IVP \eqref{eq:ivp} has a unique solution \( \eu{Y} = \eu{Y}(t;f,e) \in \fs{C}^2 ([t_0-b,t_0+b],\fs{F} ) \) for each \( (f,e) \in \fs{F} \times \fs{E} \). In addition, the mapping
	\begin{gather}\label{map:psi}
		\uppsi\colon(t_0-b,t_0+b) \times \fs{F} \times \fs{E} \to \fs{F}, \quad (t,f,e) \mapsto \eu{Y}(t;f,e)
	\end{gather}
	is of class Keller's \( C_c^1 \).
\end{theorem}
\begin{corollary}\label{cor:gift}
	Let the hypotheses of Theorem \ref{th:ivp} hold. Let \( \upphi_e\colon \bl{I} \times \fs{F} \to \fs{F} \) be the mapping defined by \( \upphi_e(t, f) = \upphi(t, f, e) \) for a fixed \( e \in \fs{E} \). If \( \varphi_1, \varphi_2 \in \fs{C}^2(\bl{I}, \fs{F}) \) are two solutions satisfying \( \varphi'_i(t) = \upphi_e(t, \varphi_i(t)) \) for \( i=1,2 \) with the same initial condition \( \varphi_1(t_0) = \varphi_2(t_0) \), then \( \varphi_1 = \varphi_2 \) on \( \bl{I} \).
\end{corollary}
\begin{proof}
	Consider the set \( S = \{ t \in \bl{I}\colon \varphi_1(t) = \varphi_2(t) \} \). By hypothesis, \( t_0 \in S \). Since \( \fs{F} \) is Hausdorff and the maps are continuous, \( S \) is closed in \( \bl{I}\). Let \( t^* \in S \). By Theorem \ref{th:ivp}, there exists a smaller interval \( \bl{I}_b = [t^* - b, t^* + b] \cap \bl{I} \) with \( b > 0 \) such that the IVP starting at \( t^* \) has a unique solution. Thus, \( \varphi_1 \) and \( \varphi_2 \) must agree on \( \bl{I}_b \), which proves that \( S \) is open in the relative topology. Hence, connectivity of \( \bl{I} \) implies \( S = \bl{I} \).
\end{proof}

\begin{corollary}\label{lem:variational_parameter}
	Let the hypotheses of Theorem \ref{th:ivp} hold. For fixed initial data $f \in \fs{F}$, the partial derivative of the solution with respect to the parameter $e$, denoted by $\mathrm{D}_3 \eu{Y}(t; f, e)$, is a Keller's $C_c^1$-mapping. For every $v \in \fs{E}$, it satisfies the following linear differential equation in $\fs{F}$
	\begin{equation}\label{eq:var_eq_frechet}
		\mathrm{D}_1 \mathrm{D}_3 \eu{Y}(t; f, e)(v) = \mathrm{D}_2 \upphi(t, \eu{Y}(t; f, e), e)(\mathrm{D}_3 \eu{Y}(t; f, e)(v)) + \mathrm{D}_3 \upphi(t, \eu{Y}(t; f, e), e)(v),
	\end{equation}
	with the initial condition $\mathrm{D}_3 \eu{Y}(t_0; f, e)(v) = \mathbf{0}_{\fs{F}}$.
\end{corollary}

\begin{proof}
	Since $\eu{Y}(t; f, e)$ is the solution, we have
	\begin{equation}\label{eq:def_ident}
		\mathrm{D}_1 \eu{Y}(t; f, e) = \upphi(t, \eu{Y}(t; f, e), e)
	\end{equation}
	for all $t$ in the existence interval. Since $\upphi$ and $\eu{Y}$ are of class Keller's $C_c^1$, we can apply the  derivative $\mathrm{D}_3$ with respect to the parameter $e$ to both sides of \eqref{eq:def_ident} in the direction $v \in \fs{E}$. On the left-hand side, we obtain
	\(
	\mathrm{D}_3 (\mathrm{D}_1 \eu{Y}(t; f, e))(v).
	\)
	By total symmetry of higher-order derivatives ([Theorem 2.4.10, Remark 2.4.1, \cite{ke}]), we have 
	\(
	\mathrm{D}_3 \mathrm{D}_1 \eu{Y}(t; f, e)(v) = \mathrm{D}_1 \mathrm{D}_3 \eu{Y}(t; f, e)(v).
	\)
	Next, we apply the chain rule ([Corollary 1.3.2, \cite{ke}]) to the right-hand side of \eqref{eq:def_ident}. Note that $e$ appears both in the second argument via the state $\eu{Y}$ and in the third argument as a direct parameter. This yields
	\begin{equation*}
		\mathrm{D}_3 (\upphi(t, \eu{Y}(t; f, e), e))(v) = \mathrm{D}_2 \upphi(t, \eu{Y}(t; f, e), e)(\mathrm{D}_3 \eu{Y}(t; f, e)(v)) + \mathrm{D}_3 \upphi(t, \eu{Y}(t; f, e), e)(v).
	\end{equation*}
	From the initial value problem, we have $\eu{Y}(t_0; f, e) = f$. Since the initial datum $f$ is independent of the parameter $e$, differentiating with respect to $e$ at $t = t_0$ gives
	\begin{equation*}
		\mathrm{D}_3 \eu{Y}(t_0; f, e)(v) = \mathrm{D}_3 (f)(v) = \mathbf{0}_{\fs{F}},
	\end{equation*}
	which concludes the proof.
\end{proof}

\section{A Frobenius Theorem}
We recall the definition of a split \fr\ submanifold (see~\cite{k6} for more details).
Let $\fs{M}$ be a Keller's $C_c^k$-Fr\'echet manifold modeled on $\fs{F} = \fs{F}_1 \oplus \fs{F}_2$ (\( k \geq 2 \)). Let \(\pi: \mathrm{T}\fs{M} \to \fs{M} \) be the tangent bundle. A subset $S \subseteq \fs{M}$ is called a split Keller's $C_c^r$-Fr\'echet submanifold modeled on $\fs{F}_1$ ($1 \leq r \leq k$), if for any $p \in S$ there exists a Keller's $ C_c^r $-diffeomorphism  $\varphi\colon U \to V = W \times O \subseteq \fs{F}_1 \times \fs{F}_2$ of $\fs{M}$ such that
\(
\varphi(S \cap U) = W \times \set{\zero{\fs{F}_2}}.
\)
The inclusion $\iota\colon S \hookrightarrow \fs{M}$ is then   a Keller's $C_c^r$-immersion, and $S$ inherits a Fr\'echet manifold structure of class  Keller's $C_c^r$.

We now present the standard definitions of distributions, integrable tangent subbundles, and the involutivity condition, following the Banach manifold approach of~\cite{lang}. A subset $\TB{D} \subseteq \TB{T}\fs{M}$ is called a \textit{distribution} on  $\fs{M}$, if for every point $m \in \fs{M}$, the set $\TB{D}_m \coloneqq  \TB{D} \cap \mathrm{T}_m \fs{M}$ is a subspace of $\mathrm{T}_m \fs{M}$. We focus on distributions that are split tangent subbundles (subbundles of the tangent bundle) of $\TB{T}\fs{M}$. A tangent subbundle \( \TB{D} \) over \( \fs{M} \) is called  integrable at \( m \in \fs{M} \) if there exists a submanifold \( S \subseteq \fs{M} \) containing \( m \) such that the inclusion \( \iota\colon S \hookrightarrow \fs{M} \) induces an isomorphism 
from \( \mathrm{T}_s S \) onto the fiber \( \TB{D}_s \), 
\( \mathrm{T}_s \iota \colon \mathrm{T}_s S \xrightarrow{\cong} \TB{D}_s \), for every \( s \in S \). We say \( \TB{D} \) is integrable if this property holds globally. From the functorial properties of vector fields and their behavior under tangent maps, it follows that any integrable subbundle must be \textit{involutive}. This requirement is formalized as the following condition:

\begin{enumerate}[label={\textup{(I)}}, ref=I]
	\item\label{itm:fr.1} The tangent subbundle \( \TB{D} \) is closed under the Lie bracket. That is, for any local sections \( \vf{V}, \vf{W} \) of \( \TB{D} \), the commutator \( [\vf{V}, \vf{W}] \) is also a section of \( \TB{D} \).
\end{enumerate}
However, satisfying this condition alone does not imply that a tangent subbundle is integrable in general. This is a well-known obstacle in Fréchet geometry, where the failure of the classical Picard-Lindelöf theorem prevents the  existence of flows for involutive distributions. Consequently, we impose a supplementary sufficient condition which is centered on the well-posedness of the associated variational problem to obtain integrability.
We now provide the local representation of the involutivity condition \ref{itm:fr.1} and describe the specific analytic form of the imposed well-posedness requirement.

Consider a product \( U \times V \) of open subsets of Fr\'echet spaces \( \fs{F} \) and \( \fs{E} \). The tangent bundle \( \mathrm{T}(U \times V) \) admits a natural decomposition as a direct sum. More precisely, for each point \( (x, y) \in U \times V \), we have the canonical isomorphism
\[
\mathrm{T}_{(x,y)}(U \times V) \cong \mathrm{T}_x(U) \times \mathrm{T}_y(V) \cong \fs{F} \times \fs{E}.
\]
It follows that the collection of fibers \( \fs{F} \times \{\mathbf{0}_{\fs{E}}\} \) forms a subbundle, denoted by \( \mathrm{T}_1(U \times V) \), which is called the first factor of the tangent bundle. By defining \( \mathrm{T}_2(U \times V) \) analogously, we obtain the Whitney sum decomposition
\[
\mathrm{T}(U \times V) = \mathrm{T}_1(U \times V) \oplus \mathrm{T}_2(U \times V).
\]
A subbundle \( \TB{D} \) of \( \mathrm{T}\fs{M} \) is {integrable} at a point \( m \in \fs{M} \) if there exists an open neighborhood \( W \) of \( m \) and a Keller's \( C_c^p \)-isomorphism \( \phi\colon U \times V \to W \) such that the tangent map \( \mathrm{T}\phi \) identifies \( \mathrm{T}_1(U \times V) \) with the restricted subbundle \( \TB{D}|_W \). Specifically, the composition 
\[
\mathrm{T}_1(U \times V) \overset{\text{inc.}}{\hookrightarrow}  \mathrm{T}(U \times V) \xrightarrow{\mathrm{T}\phi} \mathrm{T}(W)
\]
induces a bundle isomorphism of \( \mathrm{T}_1(U \times V) \) onto \( \TB{D}|_W \). That is, the following diagram is commutative 
	\begin{equation*}
		\begin{tikzcd}[column sep=3pc, row sep=3pc]
			\mathrm{T}_1(U \times V)
			\arrow[r, hook, "\text{inc}."]
			\arrow[d, "\cong", "\mathrm{T}\phi|_{\mathrm{T}_1}"']
			&
			\mathrm{T}(U \times V)
			\arrow[d, "\mathrm{T}\phi", "\cong"']
			\\
			\TB{D}|_W
			\arrow[r, hook, "\text{inc.}"]   % <-- added "inc." here
			\arrow[d, "\pi"']
			&
			\mathrm{T}W
			\arrow[d, "\pi"]
			\\
			U \times V
			\arrow[r, "\phi"']
			&
			W
		\end{tikzcd}
	\end{equation*}
 Alternatively, denote by $\phi_y\colon U \to W$ the map given by $\phi_y(x) = \phi(x, y)$, the integrability condition is satisfied if $\mathrm{T}_x\phi_y$ induces an isomorphism from $\fs{F}$ onto $\TB{D}_{\phi(x,y)}$ for all $(x, y) \in U \times V$. We note that within this local product structure, $\mathrm{T}_x\phi_y$ is precisely the partial derivative $\mathrm{D}_1\phi(x, y)$.

Given a subbundle \( \TB{D} \subseteq \mathrm{T}\fs{M} \), there exists an open neighborhood \( U \times V \) of \( (x_0, y_0) \) such that \( \TB{D} \) is represented by the exact sequence
\[
\mathbf{0}_{U \times V} \to U \times V \times \fs{F} \xrightarrow{h} U \times V \times \fs{F} \times \fs{E}.
\]
At the base point $(x_0, y_0)$, the map $h(x_0, y_0)\colon \fs{F} \to \fs{F} \times \fs{E}$ coincides with the canonical embedding from $\fs{F}$ onto $\fs{F} \times \{\zero{\fs{E}}\}$. For an arbitrary point $(x, y) \in U \times V$, the map $h(x, y)$ admits two components, $h_1(x, y)$ and $h_2(x, y)$, into $\fs{F}$ and $\fs{E}$ respectively. By applying a suitable automorphism to $U \times V \times \fs{F}$ if necessary, we may assume without loss of generality that $h_1(x, y) = \mathrm{id}_{\fs{F}}$. By setting $f(x, y) \coloneqq h_2(x, y)$, we obtain a morphism
\[
\eu{F}\colon U \times V \to \mathcal{L}_c(\fs{F}, \fs{E})
\]
of class Keller's $C_c^{p-1}$ which characterizes the subbundle completely over the neighborhood $U \times V$.
We now interpret Condition \ref{itm:fr.1} within the context of the present situation. Let $\vf{V}\colon U \times V \to \fs{F} \times \fs{E}$ be the local representation of a vector field over $U \times V$, and let $V_1$ and $V_2$ be its projections onto $\fs{F}$ and $\fs{E}$, respectively. The vector field $\vf{V}$ lies in the subbundle $\TB{D}$ if and only if
\(
V_2(x, y) = \eu{F}(x, y)(V_1(x, y))
\)
for all $(x, y) \in U \times V$. In other words, $\vf{V} \in \TB{D}$ if and only if $\vf{V}$ is of the form
\[
\vf{V}(x, y) = (V_1(x, y), \eu{F}(x, y)(V_1(x, y)))
\]
for some morphism $V_1\colon U \times V \to \fs{F}$ of class Keller's $C_c^{p-1}$. We shall also write this condition symbolically as
\begin{equation*}
	\vf{V} = (V_1, \eu{F}(V_1)).
\end{equation*}
If $\vf{V}$ and $\vf{W}$ are local representations of vector fields over $U \times V$ belonging to $\TB{D}$, then the local definition of the Lie bracket implies that $[\vf{V}, \vf{W}]$ lies in the subbundle $\TB{D}$ if and only if
\begin{equation*}
	\mathrm{D}\eu{F}(x, y)(\vf{V}(x, y))(W_1(x, y)) = \mathrm{D}\eu{F}(x, y)(\vf{W}(x, y))(V_1(x, y)),
\end{equation*}
which we write as 
\begin{equation}\label{eq:fr1}
	\mathrm{D}\eu{F}(\vf{V})(W_1) = \mathrm{D}\eu{F}(\vf{W})(V_1).
\end{equation}
Let \[ \mathcal{U}(\zero{\fs{F}}, n, \epsilon) = \Set{z \in \fs{F}}{\snorm[\fs{F}, 1]{z} < \epsilon, \dots, \snorm[\fs{F}, n]{z} < \epsilon} \] be a neighborhood of the origin \( \zero{\fs{F}} \) in \( \fs{F} \), for a sufficiently small \( \epsilon > 0 \). 
Let \( x_0 \in U \subseteq \fs{F} \) and \( y_0 \in V \subseteq \fs{E} \). We define the mapping 
\[
\upphi(t; y,z) = \eu{F}(x_0 + tz, y)(z),
\]
where \( z \in \mathcal{U} \) and \( y \in V \). To apply Theorem \ref{th:ivp}, we identify the theorem's state space with \( \fs{E} \) and its parameter space with \( \fs{F} \). Let \( \eu{I} \in \bl{AU}(\mathbb{R} \times \fs{E} \times \fs{F}, [0,\infty)) \) be a functional and assume that  Condition \ref{itm:ivp.1} holds. Specifically, for any \( \psi \in \fs{C}^2([-1,1], \fs{E}) \), \(a>0\), and each fixed \( s \in [-1,1] \), the mapping
\begin{gather*}
	\eu{J}_{\psi,s}\colon \mathbb{R} \times \fs{E} \times \fs{F} \to [0,\infty) \\
	\eu{J}_{\psi,s}(a, f, z) = \eu{I}\big(\psi'(s) - a\upphi(t_0 + as; z, \psi(s) + f)\big)
\end{gather*}
satisfies the PS-condition at any level if \( \eu{I} \) is a Keller's \( C_c^1 \)-functional, or the Chang PS-condition if \( \eu{I} \) is locally Lipschitz. Then, by Theorem \ref{th:ivp}, there exists a constant \(0 <b < a \) such that for each direction \( z \in \mathcal{U} \), there is a unique solution \( \eu{Y} = \eu{Y}(t; z) \in \fs{C}^2([t_0-b, t_0+b], \fs{E}) \) to the initial value problem
\begin{equation*}
	\begin{cases}
		\eu{Y}'(t) = \upphi(t; z, \eu{Y}(t)), \quad \forall t \in [t_0-a, t_0+a] \\
		\eu{Y}(t_0) = y.
	\end{cases}
\end{equation*}
Furthermore,
\(
D_1 	\eu{Y}(t; z, y) = \upphi(tz, \eu{Y}(t; z, y))  (z).
\)
By the conclusion of Theorem \ref{th:ivp}, the mapping \( (t, z) \mapsto \eu{Y}(t; z) \) is of class Keller's \( C_c^1 \). By performing a change of variables of the type \( t = \alpha r \) and \( z = \alpha^{-1}x \) for a small positive \( \alpha \), we may assume that the interval of existence contains \( 1 \), provided \( \epsilon \) is sufficiently small. We keep \( y \) fixed and write \( \eu{Y}(t; z) \). The differential equation is then expressed by
\begin{equation*}
	\mathrm{D}_1 \eu{Y}(t; z) = \eu{F}(x_0 + tz, \eu{Y}(t; z))(z).
\end{equation*}
We aim to prove that there exists an open neighborhood \( U_0 \) of \( x_0 \) and a unique morphism \( \eu{Z}\colon U_0 \to V \) such that
\[
\mathrm{D} \eu{Z}(x) = \eu{F}(x, \eu{Z}(x)) \quad \text{and} \quad \eu{Z}(x_0) = y_0.
\]
Consider  \( \eu{Z}(x) = \eu{Y}(1, x - x_0) \). It suffices to show that
\begin{equation}\label{eq:variational_id}
	\mathrm{D}_2 \eu{Y}(t, z)(v) = t \eu{F}(x_0 + tz, \eu{Y}(t, z))(v)
\end{equation}
for any vector \( v \in \fs{F} \). By Corollary \ref{lem:variational_parameter}, for any vector \( v \in \fs{F} \), we have
\begin{align*}
	\mathrm{D}_1 \mathrm{D}_2 \eu{Y}(t, z)(v) &= t \mathrm{D}_1 \eu{F}(x_0 + tz, \eu{Y}(t, z))(v)(z) \\
	&\quad + \mathrm{D}_2 \eu{F}(x_0 + tz, \eu{Y}(t, z))(\mathrm{D}_2 \eu{Y}(t, z)(v))(z) \\
	&\quad + \eu{F}(x_0 + tz, \eu{Y}(t, z))(v).
\end{align*}
We define the auxiliary function \[ k(t) \coloneqq \mathrm{D}_2 \eu{Y}(t, z)(v) - t \eu{F}(x_0 + tz, \eu{Y}(t, z))(v). \] and observe that \( k(0) = \zero{\fs{E}} \). Recall that the solution \( \eu{Y}(t; z) \) satisfies the  equation \( \mathrm{D}_1 \eu{Y}(t, z) = \eu{F}(x_0 + tz, \eu{Y}(t, z))(z) \). By differentiating this relation with respect to the parameter \( z \) in the direction \( v \), we obtain 
\begin{equation}\label{eq:variational_derivation}
	\mathrm{D}_1 \mathrm{D}_2 \eu{Y}(t, z)(v) = t \mathrm{D}_1 \eu{F}(x_0 + tz, \eu{Y})(v)(z) + \mathrm{D}_2 \eu{F}(x_0 + tz, \eu{Y})(\mathrm{D}_2 \eu{Y}(v))(z) + \eu{F}(x_0 + tz, \eu{Y})(v).
\end{equation}
By differentiating \( k(t) \) with respect to \( t \), substituting the identity \eqref{eq:variational_derivation}, and applying the product rule to the second term of \( k(t) \), we obtain
\begin{align*}
	k'(t) &= \mathrm{D}_1 \mathrm{D}_2 \eu{Y}(t, z)(v) - \frac{d}{dt} \left[ t \eu{F}(x_0 + tz, \eu{Y}(t, z))(v) \right] \\
	&= \left[ t \mathrm{D}_1 \eu{F}(x_0+tz, \eu{Y})(v)(z) + \mathrm{D}_2 \eu{F}(x_0+tz, \eu{Y})(\mathrm{D}_2 \eu{Y}(v))(z) + \eu{F}(x_0+tz, \eu{Y})(v) \right] \\
	&\quad - \left[ \eu{F}(x_0+tz, \eu{Y})(v) + t \mathrm{D}_1 \eu{F}(x_0+tz, \eu{Y})(z)(v) + t \mathrm{D}_2 \eu{F}(x_0+tz, \eu{Y})(\mathrm{D}_1 \eu{Y})(v) \right].
\end{align*}
Now, substituting the relations \(\mathrm{D}_1 \eu{Y} = \eu{F}(z)\) and \(\mathrm{D}_2 \eu{Y}(v) = k(t) + t \eu{F}(v)\) into the expression above yields
\begin{align*}
	k'(t) &= t \mathrm{D}_1 \eu{F}(v)(z) + \mathrm{D}_2 \eu{F}(k(t))(z) + t \mathrm{D}_2 \eu{F}(\eu{F}(v))(z) \\
	&\quad - t \mathrm{D}_1 \eu{F}(z)(v) - t \mathrm{D}_2 \eu{F}(\eu{F}(z))(v) \\
	&= \mathrm{D}_2 \eu{F}(k(t))(z) + t \left[ \mathrm{D}_1 \eu{F}(v)(z) + \mathrm{D}_2 \eu{F}(\eu{F}(v))(z) - \left( \mathrm{D}_1 \eu{F}(z)(v) + \mathrm{D}_2 \eu{F}(\eu{F}(z))(v) \right) \right].
\end{align*}
The term within the square brackets vanishes identically by  \eqref{eq:fr1}. More precisely, for the local vector fields \(\vf{V} = (v, \eu{F}(v))\) and \(\vf{W} = (z, \eu{F}(z))\), the requirement \([\vf{V}, \vf{W}] \in \TB{D}\) implies the symmetry
\[
\mathrm{D}_1 \eu{F}(v)(z) + \mathrm{D}_2 \eu{F}(\eu{F}(v))(z) = \mathrm{D}_1 \eu{F}(z)(v) + \mathrm{D}_2 \eu{F}(\eu{F}(z))(v).
\]
Consequently, \(k(t)\) satisfies the linear initial value problem
\begin{equation}\label{eq:fr5_final}
	\begin{cases}
		k'(t) = \mathrm{D}_2 \eu{F}(x_0 + tz, \eu{Y}(t, z))(k(t))(z) \\
		k(0) = \zero{\fs{E}}.
	\end{cases}
\end{equation}
The existence of the path \(k(t)\) is guaranteed on the domain of the primary solution, as it is constructed from the composition of Keller's \(C_c^1\)-morphisms \(\eu{F}\) and \(\eu{Y}\). Let \(I = [0, b]\) be the interval of existence. The linearized operator \(A(t) := \mathrm{D}_2\eu{F}(x_0+tz, \eu{Y}(t,z))(\cdot)(z)\) is continuous on \(\fs{E}\) for each \(t \in I\). Since \(I\) is compact, for each seminorm \(\Vert\cdot\Vert_{\fs{E}, i}\) in the directed family defining the topology of \(\fs{E}\), there exists a constant \(C_i > 0\) such that 
\[
\Vert A(t)(w)\Vert_{\fs{E}, i} \le C_i \Vert w\Vert_{\fs{E}, i} \quad \text{for all } t \in I \text{ and } w \in \fs{E}.
\]
Since \(k(t)\) satisfies \(k(t) = \int_0^t A(s)(k(s)) \, ds\), taking the seminorm \(\Vert\cdot\Vert_{\fs{E}, i}\) on both sides yields
\[
\Vert k(t)\Vert_{\fs{E}, i} \le \int_0^t \Vert A(s)(k(s))\Vert_{\fs{E}, i} \, ds \le C_i \int_0^t \Vert k(s)\Vert_{\fs{E}, i} \, ds.
\]
By the continuous version of Gronwall's inequality, it follows that \(\Vert k(t)\Vert_{\fs{E}, i} = 0\) for all \(t \in I\). Since this holds for every seminorm in the Hausdorff family defining the topology of \(\fs{E}\), we conclude that \(k(t) \equiv \zero{\fs{E}}\) on \(I\). This establishes the identity \(\mathrm{D}_2 \eu{Y}(t, z)(v) = t \eu{F}(x_0 + tz, \eu{Y}(t, z))(v)\), which completes the proof. The results established through the preceding local analysis are summarized in the following theorem.

\begin{theorem} \label{thm:local_frobenius}
	Let \( U \subseteq \fs{F} \) and \( V \subseteq \fs{E} \) be open subsets of Fr\'echet spaces. Let \( \eu{F}\colon U \times V \to \mathcal{L}_c(\fs{F}, \fs{E}) \) be a Keller's \( C_c^p \)-mapping, \( p \geq 2 \). Assume that
	\begin{enumerate}[label=\textup{(\alph*)}, ref=\thetheorem\textup{(\alph*)}]
		\item \textup{(Integrability)} For any \( (x, y) \in U \times V \) and all \( v, z \in \fs{F} \), we have
		\[ \mathrm{D}\eu{F}(x, y)(v, \eu{F}(x, y)v)(z) = \mathrm{D}\eu{F}(x, y)(z, \eu{F}(x, y)z)(v). \]
		\item \textup{(Well-posedness)} \label{con;well} The mapping \( \upphi(t; z, y) = \eu{F}(x_0 + tz, y)(z) \) satisfies {Condition \ref{itm:ivp.1}} for a functional \( \eu{I} \in \bl{AU}(\mathbb{R} \times \fs{F} \times \fs{E}, [0,\infty)) \).
	\end{enumerate}
	Then, for any \( (x_0, y_0) \in U \times V \), there exist neighborhoods \( U_0\) of \(x_0\) and \(V_0 \) of \(y_0\), and a unique morphism \( \eu{Z}\colon U_0 \times V_0 \to \fs{E} \) such that
	\begin{equation}
		\mathrm{D}_1 \eu{Z}(x, y) = \eu{F}(x, \eu{Z}(x, y)) \quad \text{and} \quad \eu{Z}(x_0, y) = y
	\end{equation}
	for all \( (x, y) \in U_0 \times V_0 \).
\end{theorem}
For ease of future reference and to enhance the readability of the subsequent results, we  extract the well-posedness requirement of Theorem \ref{con;well} into the following condition.

\begin{definition}[Local Well-Posedness Condition W] \label{def:condition_W}
	Let \( \fs{M} \) be a Keller's \( C_c^k \)-Fr\'echet manifold (\( k \geq 2 \)) modeled on \( \fs{F} = \fs{F}_1 \oplus \fs{F}_2 \), and let \( \TB{D} \subseteq \mathrm{T}\fs{M} \) be a split subbundle modeled on the Fr\'echet space \( \fs{F}_1 \). There exists a local product chart \( U \times V \subseteq \fs{F}_1 \times \fs{F}_2 \) and a Keller's \( C_c^{k-1} \)-morphism 
	\[
	\eu{F} \colon U \times V \to \mathcal{L}_c(\fs{F}_1, \fs{F}_2)
	\]
	such that \( \TB{D} \) is locally represented as the graph of \( \eu{F} \).  That is, a vector \( \vf{V} = (V_1, V_2) \in \fs{F}_1 \oplus \fs{F}_2 \) belongs to \( \TB{D}_{(x, y)} \) if and only if \( V_2 = \eu{F}(x, y)V_1 \). We say \( \TB{D} \) satisfies Condition W if for every point \( (x_0, y_0) \in U \times V \), the mapping 
	\[
	\upphi(t; y,z) \coloneqq \eu{F}(x_0 + tz, y)(z)
	\]
	satisfies Condition \ref{itm:ivp.1} for a functional \( \eu{I} \in \bl{AU}(\mathbb{R} \times \fs{F}_1 \times \fs{F}_2, [0,\infty)) \). 
\end{definition}

We now extend the standard definitions of integral manifolds, Frobenius charts, and distributions from Banach manifolds to the \fr\ setting, following~\cite{lang, eyn}.

Let $\TB{D} \subseteq \TB{T}\fs{M}$ be a split tangent subbundle with typical fiber  $\fs{F}_1$. A connected split  Keller's $C_c^k$-submanifold $N \subseteq \fs{M}$ is called an \textit{integral manifold} for $\TB{D}$ if 
\(
\mathrm{T}_p N = \TB{D}_p 
\)
for every $p \in N$. Given $m_0 \in \fs{M}$, an integral manifold $N$ containing $m_0$ is called \textit{maximal} if every other connected integral manifold $L$ of $\TB{D}$ containing $m_0$ is a subset of $N$ ($L \subseteq N$) and the inclusion $L \hookrightarrow N$ is of class  Keller's $ C_c^k$.

\begin{definition}
	Let \(\fs{M}\) be a Keller's \(C_c^k\)-Fr\'echet manifold, \(k \geq 2\), modeled on \(\fs{F} = \fs{F}_1 \oplus \fs{F}_2\).	Let $\TB{D} \subseteq \TB{T}\fs{M}$ be a split subbundle with typical fiber  $\fs{F}_1$. A chart $\psi\colon U \to W \times O$ of $\fs{M}$ is called a {Frobenius chart} for $\TB{D}$ if for every fixed ${\hat{y}} \in O$, the slice
	\[
	S_{\hat{y}} \coloneqq \psi^{-1}(W \times \set{\hat{y}})
	\]
	is an integral manifold of $\TB{D}$. If $\fs{M}$ admits an atlas consisting of Frobenius charts, then $\TB{D}$ is called a Frobenius distribution.
\end{definition}
\begin{proposition}[Existence of Frobenius Charts]\label{prop:efc}
	Let \(\fs{M}\) be a Keller's \(C_c^k\)-Fr\'echet manifold, \(k \geq 2\), modeled on \(\fs{F} = \fs{F}_1 \oplus \fs{F}_2\). Let \(\TB{D} \subseteq \TB{T}\fs{M}\) be a split tangent subbundle with typical fiber \(\fs{F}_1\). If \(\TB{D}\) satisfies the involutivity condition \ref{itm:fr.1} and the local well-posedness condition W in Definition~\ref{def:condition_W}, then for every \(m \in \fs{M}\), there exists a Frobenius chart \((\eu{U}, \psi)\) around \(m\).
\end{proposition}

\begin{proof}
	Let \( m \in \fs{M} \). Since \( \TB{D} \) is a split subbundle of \( \TB{T}\fs{M} \), there exists a local chart \( (U, \phi) \) around \( m \) such that \( \phi(U) = W \times O \subseteq \fs{F}_1 \times \fs{F}_2 \) with \(\phi(m) = (x_0, y_0)\). In this chart, the distribution \( \TB{D} \) is represented locally by the graph of a Keller's \( C_c^k \)-morphism \( \eu{F}\colon W \times O \to \mathcal{L}_c(\fs{F}_1, \fs{F}_2) \).
	
	By Theorem~\ref{thm:local_frobenius}, there exist open neighborhoods \( W_0 \subseteq W \) and \( O_0 \subseteq O \) around \(x_0\) and \(y_0\), and a unique morphism \( \eu{Z}\colon W_0 \times O_0 \to \fs{F}_2 \) such that
	\begin{equation}\label{eq:local_ode}
		\mathrm{D}_1 \eu{Z}(x, y) = \eu{F}(x, \eu{Z}(x, y)) \quad \text{and} \quad \eu{Z}(x_0, y) = y
	\end{equation}
	for all \( (x, y) \in W_0 \times O_0 \). Define the mapping \( \upphi\colon W_0 \times O_0 \to W \times O \) by \( \upphi(x, y) = (x, \eu{Z}(x, y)) \). Its derivative \( \mathrm{D}\upphi(x, y) \) applied to a vector \( (u, v) \in \fs{F}_1 \times \fs{F}_2\) is given by
	\[
	\mathrm{D}\upphi(x, y) ((u, v)) = \big(u, \mathrm{D}_1 \eu{Z}(x, y) (u) + \mathrm{D}_2 \eu{Z}(x, y) (v)\big).
	\]
	By evaluating this at the base point \( (x_0, y_0) \), we note that since \( \eu{Z}(x_0, y) = y \) for all \( y \in O_0 \), it follows that \( \mathrm{D}_2 \eu{Z}(x_0, y_0) = \mathrm{id}_{\fs{F}_2} \). Additionally, \( \mathrm{D}_1 \eu{Z}(x_0, y_0) = \eu{F}(x_0, y_0) \). Therefore, the derivative \( \mathrm{D}\upphi(x_0, y_0) \) is given by the isomorphism
	\[
	\mathrm{D}\upphi(x_0, y_0) = \begin{pmatrix} \mathrm{id}_{\fs{F}_1} & \zero{\fs{F}_2} \\ \eu{F}(x_0, y_0) & \mathrm{id}_{\fs{F}_2} \end{pmatrix}.
	\]
To establish that \(\upphi\) is a local Keller's \(C^k_c\)-isomorphism at \((x_0, y_0)\), we explicitly construct its local inverse using the well-posedness guaranteed by Condition~W (Definition~\ref{def:condition_W}). Let \((u, v)\) be in an open neighborhood of \(\upphi(x_0, y_0)\). The equation \(\upphi(x, y) = (u, v)\) is equivalent to the following:
\[
x = u \quad \text{and} \quad \eu{Z}(u, y) = v.
\]
By Condition~W, the curve \(\eu{Z}(u, y)\) is defined via the unique solution of the parameter-dependent initial value problem. Due to the uniqueness of these integral curves guaranteed by Corollary~\ref{cor:gift} and their smooth dependence on initial states and parameters as established in Theorem~\ref{th:ivp}, the equation \(\eu{Z}(u, y) = v\) can be uniquely solved for \(y\) as \(y = \Psi(u, v)\), where \(\Psi\) represents the parameter-inverted solution map. Therefore, the local inverse map is explicitly defined by
\[
\upphi^{-1}(u, v) = (u, \Psi(u, v)).
\]
Since the well-posedness under Condition~W  asserts via Theorem~\ref{th:ivp} that the solution mapping \(\uppsi\) in \eqref{map:psi} is Keller's \(C^k_c\)-smooth with respect to all variables, the parameter-inverted map \(\Psi\) (and consequently \(\upphi^{-1}\)) is a Keller \(C^k_c\)-map. Direct verification yields \(\upphi \circ \upphi^{-1} = \mathrm{id}\) and \(\upphi^{-1} \circ \upphi = \mathrm{id}\) on their respective local neighborhoods. Thus, \(\upphi\) is a local Keller's \(C^k_c\)-isomorphism at \((x_0, y_0)\).
	
	We now define the new chart \((\eu{U}, \psi)\) around \(m\) by 
	\[
	\psi \coloneqq \upphi^{-1} \circ \phi|_{\eu{U}},
	\]
	where \(\eu{U} = \phi^{-1}(\upphi(W_0 \times O_0))\). This map is a composition of Keller's \(C_c^k\)-diffeomorphisms and is therefore a well-defined chart for \(\fs{M}\).
	
	To verify that \(\psi\) is a Frobenius chart, we examine the slices \(S_{\hat{y}} \coloneqq \psi^{-1}(W_0 \times \{\hat{y}\})\) for a fixed \(\hat{y} \in O_0\). By the definition of \(\psi\), we have
	\[
	\phi(S_{\hat{y}}) = \upphi(W_0 \times \{\hat{y}\}) = \{ (x, \eu{Z}(x, \hat{y})) \mid x \in W_0 \__.
	\]
	The tangent space of this slice at any point \(p \in S_{\hat{y}}\) (expressed in the local \(\phi\)-chart) is the image of \(\fs{F}_1\) under the derivative of the mapping \(x \mapsto (x, \eu{Z}(x, \hat{y}))\), which is given by
	\[
	\mathrm{T}_{\phi(p)} \phi(S_{\hat{y}}) = \{ (v, \mathrm{D}_1 \eu{Z}(x, \hat{y}) (v)) \mid v \in \fs{F}_1 \}.
	\]
	By substituting the relation from \eqref{eq:local_ode}, we obtain
	\[
	\mathrm{T}_{\phi(p)} \phi(S_{\hat{y}}) = \{ (v, \eu{F}(x, \eu{Z}(x, \hat{y})) (v)) \mid v \in \fs{F}_1 \},
	\]
	which is precisely the local representation of the fiber \(\TB{D}_p\). Thus, \(\mathrm{T}_p S_{\hat{y}} = \TB{D}_p\), implying that each slice \(S_{\hat{y}}\) is an integral manifold of \(\TB{D}\). Therefore, \((\eu{U}, \psi)\) is a Frobenius chart.
\end{proof}

\begin{remark}\label{rem:1}
	We prove that the chart \(\psi = \upphi^{-1} \circ \phi\) satisfies \(\mathrm{T}\psi(\TB{D}|_{\eu{U}}) = \fs{F}_1 \times \{\zero{\fs{F}_2}\}\). Let \(p \in \eu{U}\) and let \((x, y) = \phi(p)\). By definition of the distribution in the chart \(\phi\), any tangent vector \(\omega \in \TB{D}_p\) is given by
	\begin{equation} \label{eq:dist_orig}
		\mathrm{T}_p\phi(\omega) = (v, \eu{F}(x, y)v) \in \fs{F}_1 \oplus \fs{F}_2, \quad \text{for some } v \in \fs{F}_1.
	\end{equation}
	Let \((x, \hat{y}) = \psi(p)\). By the chain rule, the tangent map of the change of coordinates is \(\mathrm{T}_p\psi = (\mathrm{D}\upphi{(x, \hat{y})})^{-1} \circ \mathrm{T}_p\phi\). The derivative of the map \(\upphi(x, \hat{y}) = (x, \eu{Z}(x, \hat{y}))\) is given by 
	\begin{equation*}
		\mathrm{D}\upphi(x, \hat{y}) = \begin{pmatrix} \id_{\fs{F}_1} & \zero{\fs{F}_2} \\ \mathrm{D}_1 \eu{Z}(x, \hat{y}) & \mathrm{D}_2 \eu{Z}(x, \hat{y}) \end{pmatrix}.
	\end{equation*}
	Thus the action of \(\mathrm{D}\upphi\) on a  vector \((v, \zero{\fs{F}_2})\) is as follows
	\begin{align*}
		\mathrm{D}\upphi(x, \hat{y})  \begin{pmatrix} v \\ \zero{\fs{F}2} \end{pmatrix} &= \begin{pmatrix} \id_{\fs{F}_1} & \zero{\fs{F}_2} \\ \mathrm{D}_1 \eu{Z}(x, \hat{y}) & \mathrm{D}_2 \eu{Z}(x, \hat{y}) \end{pmatrix} \begin{pmatrix} v \\ \zero{\fs{F}_2} \end{pmatrix} \\
		&= (v, \mathrm{D}_1 \eu{Z}(x, \hat{y})(v)).
	\end{align*}
	By  Theorem \ref{thm:local_frobenius}, \(\eu{Z}\) satisfies  \(\mathrm{D}_1 \eu{Z}(x, \hat{y}) = \eu{F}(x, \eu{Z}(x, \hat{y}))\). Since \(\phi(p) = \upphi(x, \hat{y}) = (x, \eu{Z}(x, \hat{y}))\), substituting this into the expression above yields
	\begin{equation*}
		\mathrm{D}\upphi(x, \hat{y})  ((v, \zero{\fs{F}_2})) = (v, \eu{F}(\phi(p))v).
	\end{equation*}
	By comparing this with Equation \eqref{eq:dist_orig}, we see that \(\mathrm{D}\upphi(x, \hat{y}) 
	((v, \zero{\fs{F}_2})) = \mathrm{T}_p\phi(\omega)\), and by multiplying by the inverse operator on both sides, we obtain
	\begin{equation*}
		(v, \zero{\fs{F}_2}) = (\mathrm{D}\upphi(x, \hat{y}))^{-1} \circ \mathrm{T}_p\phi(\omega) = \mathrm{T}_p\psi(\omega).
	\end{equation*}
	Since this holds for all \(\omega \in \TB{D}_p\), it follows that \(\mathrm{T}_p\psi(\TB{D}_p) = \fs{F}_1 \times \{\zero{\fs{F}_2}\}\).
\end{remark}
\begin{theorem}[Global Frobenius Theorem for Fr\'echet Manifolds] \label{thm:global_frobenius}
	Let \(\fs{M}\) be a Keller's \(C_c^k\)-Fr\'echet manifold, \(k \geq 2\), modeled on \(\fs{F} = \fs{F}_1 \oplus \fs{F}_2\). Let \(\TB{D} \subseteq \TB{T}\fs{M}\) be a split tangent subbundle with typical fiber \(\fs{F}_1\). If \(\TB{D}\) satisfies the involutivity condition \ref{itm:fr.1} and the local well-posedness condition
	W in Definition~\ref{def:condition_W}, then \( \fs{M} \) admits an atlas of Frobenius charts for \( \TB{D} \), and for every point \( m_0 \in \fs{M} \), there exists a unique maximal connected integral manifold \( L_{m_0} \) passing through \( m_0 \).
\end{theorem}
\begin{proof}
	By Proposition \ref{prop:efc} and Remark \ref{rem:1}, for every point \(m \in \fs{M}\), there exists a coordinate chart \((\eu{U}, \psi)\) such that \(\psi(\eu{U}) = W \times O \subseteq \fs{F}_1 \times \fs{F}_2\) and 
	\( \mathrm{T}\psi(\TB{D}|_{\eu{U}}) = \fs{F}_1 \times \{\zero{\fs{F}_2}\} \).
	
	Let \( P_c^k([0, 1], \fs{M})\) be the set of all piecewise Keller's \(C_c^k\)-curves \(\gamma\colon [0, 1] \to \fs{M}\). 
	A curve \(\gamma \in P_c^k([0, 1], \fs{M}) \) is called \(\TB{D}\)-tangent if \(\gamma'(t) \in \TB{D}_{\gamma(t)}\) holds for almost every \(t \in [0, 1]\). Define the leaf passing through \( m_0 \), denoted by \( L_{m_0} \), as follows
	\[
	L_{m_0} \coloneqq \left\{ p \in \fs{M} \mid \exists \gamma \in P_c^k([0, 1], \fs{M}) \colon \gamma(0)=m_0, \, \gamma(1)=p, \text{ and } \gamma'(t) \in \TB{D}_{\gamma(t)} \text{ for a.e. } t \in [0, 1] \right\}.
	\]
	By Proposition \ref{prop:efc}, there exists an atlas of Frobenius charts \( \mathcal{A} = \{ (\eu{U}_i, \psi_i) \}_{i \in I} \) for \( \fs{M} \). For any \( p \in L_{m_0} \), let \( (\eu{U}, \psi) \) be a Frobenius chart around \( p \) with \( \psi(p) = (x_p, y_p) \in W \times O \). 
	
	Since \(\fs{F}_1\) is locally path-connected, we can choose a path-connected open neighborhood \(W_0 \subseteq W\) containing \(x_p\). The restricted local slice passing through \( p \) is defined by
	\[ S_{p,0} \coloneqq \psi^{-1}(W_0 \times \{y_p\}). \]
	Now we show that if \(p \in L_{m_0}\), then \( S_{p,0} \subseteq L_{m_0}\). Let \(q \in S_{p,0}\). Since \(S_{p,0} \cong W_0 \subseteq \fs{F}_1\) is path-connected by construction, there exists a Keller's \(C_c^k\)-curve \(\sigma\colon [0,1] \to S_{p,0}\) with \(\sigma(0)=p\) and \(\sigma(1)=q\).	Because \(S_{p,0}\) is an integral manifold, \(\sigma'(t) \in \mathrm{T}_{\sigma(t)} S_{p,0} = \TB{D}_{\sigma(t)}\). Since \(p \in L_{m_0}\), there exists a \(\TB{D}\)-tangent piecewise Keller's \(C_c^k\)-curve \(\gamma\) from \(m_0\) to \(p\). The concatenation \(\eta = \gamma \ast \sigma\) defined by 
	\[
	(\gamma \ast \sigma)(t) = 
	\begin{cases} 
		\gamma(2t) & t \in [0, 1/2] \\
		\sigma(2t-1) & t \in (1/2, 1]
	\end{cases}
	\]
	is a piecewise Keller's \(C_c^k\)-curve from \(m_0\) to \(q\) that is \(\TB{D}\)-tangent. By the definition of \(L_{m_0}\), we conclude \(q \in L_{m_0}\), and thus \(S_{p,0} \subseteq L_{m_0}\).
	
	We endow \( L_{m_0} \) with the leaf topology, denoted by \( \tau_{\ell} \). The collection of all path-connected local slices 
	\[ \mathcal{B} = \{ S_{p,0} \mid p \in L_{m_0}, S_{p,0} \text{ is a path-connected local slice of a Frobenius chart} \} \]
	serves as a basis for \( \tau_{\ell} \). A subset \( \eu{O} \subseteq L_{m_0} \) is open in \( \tau_{\ell} \) if and only if it is a union of such path-connected slices. Because the transition maps between Frobenius charts are diffeomorphisms that preserve the distribution, they locally map slices to slices. Consequently, the intersection of two path-connected slices \( S_{p,0} \cap S_{q,0} \) corresponds to open sets in the coordinate spaces of both slices, meaning it is open in the intrinsic manifold topology of each individual slice. This implies that \( \mathcal{B} \) satisfies the basis axioms. The leaf topology \( \tau_{\ell} \) is strictly finer than the subspace topology, as the slices do not necessarily contain all points of \( L_{m_0} \) found in a neighborhood of \( \fs{M} \). Therefore, the inclusion map \( \iota \colon (L_{m_0}, \tau_{\ell}) \hookrightarrow (\fs{M}, \tau_{\fs{M}}) \) is continuous. 
	
	Now we prove that \( (L_{m_0}, \tau_{\ell}) \) is a Hausdorff topological space. Let \( p, q \) be two distinct points in \( L_{m_0} \). Since \( L_{m_0} \subseteq \fs{M} \) and \( \fs{M} \) is a Hausdorff space, there exist disjoint open sets \( U, V \in \tau_{\fs{M}} \) such that \( p \in U \) and \( q \in V \). The topology \( \tau_{\ell} \) is finer than the subspace topology \( \tau_{\text{sub}} \) inherited from \( \fs{M} \). Specifically, the inclusion map \( \iota\colon (L_{m_0}, \tau_{\ell}) \to (\fs{M}, \tau_{\fs{M}}) \) is continuous. Consequently, the sets 
	\[ O_p \coloneqq \iota^{-1}(U) = U \cap L_{m_0} \quad \text{and} \quad O_q \coloneqq \iota^{-1}(V) = V \cap L_{m_0} \]
	are open in the leaf topology \( \tau_{\ell} \). Since \( U \cap V = \emptyset \), it follows that \( O_p \cap O_q = \emptyset \). Thus, \( L_{m_0} \) is Hausdorff. 
	
	For each \( p \in L_{m_0} \), we define the local chart \( \alpha_p \coloneqq \text{pr}_1 \circ \psi|_{S_{p,0}} \colon S_{p,0} \to W_0 \subseteq \fs{F}_1 \). We prove that \( L_{m_0} \) is a Keller's \( C_c^k \)-Fr\'echet manifold. Let \( S_{p,0} \) and \( S_{q,0} \) be two overlapping path-connected slices. The transition map is given by \( \tau_{qp} = \alpha_q \circ \alpha_p^{-1} \). Let \( \Psi_{qp} = \psi_q \circ \psi_p^{-1} \) be the transition map between the corresponding Frobenius charts of \( \fs{M} \), which we write in component form as \( \Psi_{qp}(x, y) = (\xi(x, y), \eta(x, y)) \). Since both charts are Frobenius, the derivative \( \mathrm{D}\Psi_{qp} \) maps the subspace \( \fs{F}_1 \times \{\zero{\fs{F}_2}\} \) to itself. This implies that the partial derivative with respect to the first factor of the second component vanishes, i.e., \( \mathrm{D}_1 \eta = 0 \). Consequently, on each connected component of the overlap, \( \eta \) depends only on \( y \). On a specific leaf where \( y \) is fixed, the transition map simplifies to the following
	\[ \tau_{qp}(x) = \xi(x, \text{const}). \]
	As \( \Psi_{qp} \) is a Keller's \( C_c^k \)-diffeomorphism of \( \fs{M} \), its restriction \( \tau_{qp} \) is a \( C_c^k \)-diffeomorphism between open sets of \( \fs{F}_1 \). Thus, the collection of charts \( \{ (S_{p,0}, \alpha_p) \}_{p \in L_{m_0}} \) defines a Keller's \( C_c^k \)-atlas for \( L_{m_0} \), modeled on the Fr\'echet space \( \fs{F}_1 \).
	
	We now verify that \( L_{m_0} \) is an integral manifold of the distribution \( \TB{D} \). By the definition of the topology, for any point \( p \in L_{m_0} \), the path-connected local slice \( S_{p,0} \) is an open neighborhood of \( p \) in \( L_{m_0} \). Consequently, the tangent space of the leaf at \( p \) is identical to the tangent space of the slice, i.e.,
	\(
	\mathrm{T}_p L_{m_0} = \mathrm{T}_p S_{p,0}.
	\)
	Recall that the restricted slice \( S_{p,0} \) is defined through a Frobenius chart \( (\eu{U}, \psi) \) as the pre-image \( \psi^{-1}(W_0 \times \{ y_p \}) \). By Remark \ref{rem:1}, the tangent map \( \mathrm{T}_p\psi \) satisfies the following
	\begin{equation*}
		\mathrm{T}_p\psi(\TB{D}_p) = \fs{F}_1 \times \{\zero{\fs{F}_2}\}.
	\end{equation*}
	Since the tangent space to the slice \( W_0 \times \{ y_p \} \) in the coordinate space is precisely \( \fs{F}_1 \times \{\zero{\fs{F}_2}\} \), the inverse map \( (\mathrm{T}_p\psi)^{-1} \) carries this subspace back to the distribution. Thus,
	\(
	\mathrm{T}_p S_{p,0} = \TB{D}_p
	\). Combining these yields \( \mathrm{T}_p L_{m_0} = \TB{D}_p \) for all \( p \in L_{m_0} \), thus \( L_{m_0} \) is an integral manifold of \( \TB{D} \).
	
	Let \( N \) be any connected integral manifold of \( \TB{D} \) such that \( m_0 \in N \). We show that \( N \) is necessarily a submanifold of \( L_{m_0} \). Let \( q \in N \). Since \( N \) is connected and locally path-connected, it follows that \( N \) is path-connected. Therefore, there exists a piecewise Keller's \( C_c^k \)-curve \( \gamma\colon [0, 1] \to N \) connecting \( m_0 \) to \( q \). Because \( N \) is an integral manifold, its tangent bundle satisfies \( \mathrm{T} N = \TB{D}|_N \). Consequently, \( \gamma'(t) \in \TB{D}_{\gamma(t)} \) holds almost everywhere. By the definition of the leaf, we have \( q \in L_{m_0} \), thus \( N \subseteq L_{m_0} \).
	
	Now, consider a point \( p \in N \) and a Frobenius chart \( (\eu{U}, \psi) \) around \( p \). In this chart, \( \TB{D} \) is represented by the subspace \( \fs{F}_1 \times \{\zero{\fs{F}_2}\} \). Since \( \mathrm{T} N = \TB{D} \), the image \( \psi(N \cap \eu{U}) \) must be tangent to these slices. For a connected \( N \), this local constraint implies that \( N \cap \eu{U} \) lies within a single path-connected slice \( S_{p,0} \). Because \( S_{p,0} \) is an open set in the topology of \( L_{m_0} \), the inclusion map \( \iota\colon N \hookrightarrow L_{m_0} \) is locally an identification with a slice. This implies that \( \iota \) is a Keller's \( C_c^k \)-immersion. Therefore, \( L_{m_0} \) is the unique maximal connected integral manifold containing \( m_0 \), and any other such manifold is an open submanifold of \( L_{m_0} \).
	
\end{proof}

We adopt the following definition from the theory of Banach manifolds (see, e.g., \cite{ab}).
\begin{definition}[Foliation of a Fr\'echet Manifold] \label{def:foliation}
	Let \(\fs{M}\) be a Keller's \(C_c^k\)-Fr\'echet manifold modeled on \(\fs{F} = \fs{F}_1 \oplus \fs{F}_2\). A foliation \(\Phi = \{\mathcal{L}_\alpha\}_{\alpha \in A}\) of \(\fs{M}\) is a partition of \(\fs{M}\) into disjoint connected sets, called leaves, such that every point \(m \in \fs{M}\) has a chart \((\eu{U}, \psi)\) with \(\psi(\eu{U}) = W \times O \subseteq \fs{F}_1 \times \fs{F}_2\) such that for each leaf \(\mathcal{L}_\alpha\), every connected component of the intersection \(\eu{U} \cap \mathcal{L}_\alpha\), denoted by \((\eu{U} \cap \mathcal{L}_\alpha)^\beta\), is given by
	\[
	\psi((\eu{U} \cap \mathcal{L}_\alpha)^\beta) = W \times \{c_\alpha^\beta\}
	\]
	for some constant \(c_\alpha^\beta \in O\). Such a chart is called foliated or distinguished by \(\Phi\). The dimension (resp., codimension) of the foliation is the dimension of \(\fs{F}_1\) (resp., \(\fs{F}_2\)). The set 
	\[
	\TB{D}(\Phi) \coloneqq \bigcup_{\alpha \in A} \bigcup_{m \in \mathcal{L}_\alpha} \mathrm{T}_m \mathcal{L}_\alpha
	\]
	is a split subbundle of \( \mathrm{T}\fs{M} \) called the tangent bundle to the foliation. The quotient bundle 
	\(
	\nu(\Phi) \coloneqq \mathrm{T}\fs{M} / \TB{D}(\Phi)
	\)
	is called the normal bundle to the foliation \( \Phi \). Elements of \( \TB{D}(\Phi) \) are called vectors tangent to the foliation \( \Phi \).
\end{definition}
Let \( (\eu{U}, \psi) \) be a foliated chart as defined in Definition~\ref{def:foliation}, where \( \psi \colon \eu{U} \to W \times O \subseteq \fs{F}_1 \times \fs{F}_2 \). For any point \( u \in \eu{U} \cap L_\alpha \), the leaf component \( (\eu{U} \cap L_\alpha)^\beta \) is mapped by \( \psi \) to the slice \( W \times \{c_\alpha^\beta\} \). Since the tangent space to a slice in the model space is simply the first factor, we have
\[
\mathrm{T}_u \psi(\mathrm{T}_u \mathcal{L}_\alpha) = \fs{F}_1 \times \{\zero{\fs{F}_2}\}.
\]
Consequently, the restriction of the tangent bundle of the foliation to the chart neighborhood \( \eu{U} \) satisfies
\[
\mathrm{T}\psi\bigl(\TB{D}(\Phi)|_{\eu{U}}\bigr) = (W \times O) \times \bigl(\fs{F}_1 \times \{\zero{\fs{F}_2}\}\bigr).
\]
The tangent bundle charts induced by the foliated atlas of \( \fs{M} \) satisfy the subbundle property by mapping each vector \( v_m \in \TB{D}(\Phi)_m \) to the pair \( (\psi(m), \mathrm{T}_m \psi(v_m)) \in (W \times O) \times (\fs{F}_1 \times \{\zero{\fs{F}_2}\}) \). Because the fiber \( \fs{F}_1 \times \{\zero{\fs{F}_2}\} \) is a split subspace of \( \fs{F}_1 \oplus \fs{F}_2 \), the bundle \( \TB{D}(\Phi) \) is a split subbundle of \( \TB{T}\fs{M} \). It follows that the quotient \( \nu(\Phi)_m \cong \fs{F}_2 \) is well-defined and inherits the structure of a Fr\'echet vector bundle over \( \fs{M} \).
\begin{lemma}[Slicing Property of Frobenius Charts] \label{lem:slicing}
	Let \(\psi\colon \eu{U} \to W \times O \subseteq \fs{F}_1 \times \fs{F}_2\) be a Frobenius chart for the distribution \(\TB{D}\). Let \(S \subseteq \eu{U}\) be a connected Keller's \(C_c^k\)-submanifold, \(k\geq1\), modeled on 
	\(\fs{F}_1\) such that for every \(p \in S\), the tangent space satisfies \(\mathrm{T}_p \psi(\mathrm{T}_p S) \subseteq \fs{F}_1 \times \{\zero{\fs{F}_2}\}\). Then there exists a unique constant \(c \in O\) such that \(\psi(S) \subseteq W \times \{c\}\).
\end{lemma}

\begin{proof}
	Let \(\hat{S} = \psi(S) \subseteq W \times O\). Since \(\psi\) is a homeomorphism and \(S\) is connected, \(\hat{S}\) is a connected subset of the product Fr\'echet space \(\fs{F}_1 \times \fs{F}_2\), and  since \(\hat{S}\) is locally path-connected, it follows that \(\hat{S}\) is path-connected. 
	
	Let \(\mathrm{Pr}_2\colon \fs{F}_1 \times \fs{F}_2 \to \fs{F}_2\) be the canonical projection onto the second factor, \(\mathrm{Pr}_2(x, y) = y\). This is a continuous linear map, and its derivative at any point \((x, y)\) is simply the map itself, i.e., \(\mathrm{D}\mathrm{Pr}_2(x, y)(h_1, h_2) = h_2\). Now, take any two points \(a, b \in \hat{S}\). There exists a piecewise Keller's \(C_c^k\)-curve \(\gamma\colon [0, 1] \to \hat{S}\) such that \(\gamma(0) = a\) and \(\gamma(1) = b\). By the Fundamental Theorem of Calculus ([Proposition 1.1.3, \cite{ke}]) we have
	\[
	\mathrm{Pr}_2(b) - \mathrm{Pr}_2(a) = \int_0^1 \mathrm{D}\mathrm{Pr}_2(\gamma(t))  (\gamma'(t)) \, dt.
	\]
	By the hypothesis, the tangent vector to the curve \(\gamma'(t)\) must lie in \(\mathrm{T}_{\gamma(t)}\hat{S}\). Therefore, \(\gamma'(t) \in \fs{F}_1 \times \{\zero{\fs{F}_2}\}\) for almost every \(t\). Let \(\gamma'(t) = (v_1(t), \zero{\fs{F}_2})\). Substituting this into the derivative of the projection implies
	\[
	\mathrm{D}\mathrm{Pr}_2(\gamma(t))  ((v_1(t), \zero{\fs{F}_2})) = \mathrm{Pr}_2(v_1(t),\zero{\fs{F}_2}) =\zero{\fs{F}_2}.
	\]
	The integrand vanishes identically on \([0, 1]\), thus
	\(
	\mathrm{Pr}_2(b) = \mathrm{Pr}_2(a)
	\).	Since this holds for any \(a, b \in \hat{S}\), the second coordinate must be constant across the entire set \(\hat{S}\).
\end{proof}
\begin{proposition}
	Under the hypotheses of Theorem \ref{thm:global_frobenius}, the collection of unique maximal connected integral manifolds \(\Phi = \{L_m\}_{m \in \fs{M}}\) constitutes a foliation of \(\fs{M}\).
\end{proposition}

\begin{proof}
	By Theorem \ref{thm:global_frobenius}, for every \(m \in \fs{M}\), the leaf \(L_m\) is defined as the set of points that can be joined to \(m\) by piecewise \(\TB{D}\)-tangent Keller's \(C_c^k\)-curves. This relation defines an equivalence relation on \(\fs{M}\) where reflexivity is established by the existence of constant curves \(\gamma(t) = m\) at every point. Symmetry is guaranteed because if \(p\) is connected to \(m\), then the time-reversed curve \(\gamma(1-t)\) is also \(\TB{D}\)-tangent, showing \(m\) is connected to \(p\). Transitivity follows from the concatenation of curves \(\gamma \ast \sigma\) as defined in the proof of Theorem \ref{thm:global_frobenius}, implying that if \(p \in L_m\) and \(q \in L_p\), then \(q \in L_m\).	As equivalence classes, the leaves \(L_m\) are either disjoint or identical, and their union is \(\fs{M}\). Connectedness follows directly from the path-reachability definition.
	
	Let \(m \in \fs{M}\) and choose a Frobenius chart \((\eu{U}, \psi)\) around \(m\) such that \(\mathrm{T}\psi(\TB{D}|_{\eu{U}}) = \fs{F}_1 \times \{\zero{\fs{F}_2}\}\). Let \(L_\alpha\) be a leaf and \((\eu{U} \cap L_\alpha)^\beta\) be a connected component of its intersection with the chart neighborhood. 
	We first show that \((\eu{U} \cap L_\alpha)^\beta\) is open in the leaf topology \(\tau_\ell\). For any point \(p \in (\eu{U} \cap L_\alpha)^\beta\), let \( y = (\mathrm{Pr}_2 \circ \psi)(p)\) and set \(S_p = \psi^{-1}(W \times \{y\})\) be the local slice containing \(p\). By the construction of the leaf topology, \(S_p\) is an open neighborhood of \(p\) in \((L_\alpha, \tau_\ell)\) and satisfies \(S_p \subseteq \eu{U} \cap L_\alpha\). Since \(S_p \cong W\) is path-connected (and thus connected), it must be contained within the same connected component of the intersection, i.e., \(S_p \subseteq (\eu{U} \cap L_\alpha)^\beta\). Thus, every point in the component has a \(\tau_\ell\)-open neighborhood inside the component, proving it is open in \(\tau_\ell\).
	
	Since \((\eu{U} \cap L_\alpha)^\beta\) is an open subset of the integral manifold \(L_\alpha\), their tangent spaces coincide at each point. Thus, for each \(p \in (\eu{U} \cap L_\alpha)^\beta\), the tangent space satisfies \(\mathrm{T}_p (\eu{U} \cap L_\alpha)^\beta = \TB{D}_p\). In the Frobenius chart, this corresponds to the following
	\[ 
	\mathrm{T}_p \psi \left( \mathrm{T}_p (\eu{U} \cap L_\alpha)^\beta \right) = \fs{F}_1 \times \{\zero{\fs{F}_2}\}.
	\]
	By Lemma \ref{lem:slicing}, there exists a unique \(c_\alpha^\beta \in O\) such that the image of the component is restricted to a constant \(\fs{F}_2\)-coordinate, i.e.,
	\[
	\psi((\eu{U} \cap L_\alpha)^\beta) \subseteq W \times \{c_\alpha^\beta\}.
	\]
	Conversely, for any point \(p \in (\eu{U} \cap L_\alpha)^\beta\), the slice \(S_p = \psi^{-1}(W \times \{c_\alpha^\beta\})\) (which is the same for all points because the second coordinate is constant) is contained in the component \((\eu{U} \cap L_\alpha)^\beta\), as shown when proving openness. We therefore conclude the equality
	\[
	\psi((\eu{U} \cap L_\alpha)^\beta) = W \times \{c_\alpha^\beta\}.
	\]
	This means that \((\eu{U}, \psi)\) is a foliated chart for the collection \(\Phi\), and completes the proof.
\end{proof}

\begin{theorem}[The Global Frobenius and Foliation Theorem] \label{thm:global_frobenius_foliation}
	Let \(\fs{M}\) be a Keller's \(C_c^k\)-Fr\'echet manifold \textup{(}\(k \geq 2\)\textup{)} and \(\TB{D} \subseteq \TB{T}\fs{M}\) a split subbundle that satisfies the local well-posedness condition W at every point. Then the following are equivalent\(\colon\)
	\begin{enumerate}[label=\textup{(\roman*)}, leftmargin=3em]
		\item There exists a foliation \(\Phi\) on \(\fs{M}\) such that \(\TB{D} = \TB{D}(\Phi)\).
		\item \(\TB{D}\) is integrable \textup{(}every point \(m \in \fs{M}\) is contained in a unique maximal connected integral manifold \(L_m\)\textup{)}.
		\item \(\TB{D}\) is involutive.
	\end{enumerate}
\end{theorem}

\begin{proof}
	\textup{(i) \(\implies\) (ii):} The proof is standard; the leaves of the foliation serve as the unique maximal connected integral manifolds.
	
	\textup{(ii) \(\implies\) (iii):} Let \(X, Y\) be local sections of \(\TB{D}\). For any \(m \in \fs{M}\), let \(L_m\) be the unique integral manifold through \(m\). Since \(X\) and \(Y\) are tangent to \(L_m\), their restriction to \(L_m\) are well-defined vector fields on the submanifold. Their Lie bracket \([X, Y]\) restricted to \(L_m\) is therefore tangent to \(L_m\), which implies \([X|_m, Y|_m] \in \TB{D}_m\).
	
	\textup{(iii) \(\implies\) (i):} This follows directly from Theorem \ref{thm:global_frobenius_foliation}. Since \(\TB{D}\) is involutive and satisfies Condition W by the standing hypothesis, it admits Frobenius charts, which define a foliation \(\Phi\) such that \(\TB{D} = \TB{D}(\Phi)\).
\end{proof}

For the convenience of the reader, we present the following definitions, following the treatment of general locally convex manifolds given in \cite{neeb}. 
\begin{definition}
	Let \( \fs{M} \) be a Keller's \( C_c^k \)-Fr\'echet manifold modeled on \(\fs{F}\) and \( \fs{E} \) be a Fr\'echet space. We denote by \(\eu{V}_{C_c}^k(\fs{M}, E)\) the set of all Keller's \(C_c^k\) sections of a vector bundle \(E \to \fs{M}\). 
	When \(E = \mathrm{T}\fs{M}\), we write \(\eu{V}_{C_c}^k(\fs{M})\) instead of \(\eu{V}_{C_c}^k(\fs{M}, \mathrm{T}\fs{M})\).
	For an open set \(\eu{U} \subseteq \fs{M}\), \(\eu{V}_{C_c}^k(\eu{U})\) denotes the space of \(C_c^k\) sections of the restricted tangent bundle \(\mathrm{T}\eu{U}\).

	(a) An {\( \fs{E} \)-valued  Keller's \( C_c^k \)-\( p \)-form} \( \omega \) on \( \fs{M} \) is a mapping that associates to each \( x \in \fs{M} \) a \( p \)-linear alternating map \( \omega_x\colon (\mathrm{T}_x{\fs{M}})^p \to \fs{E} \). We require that for every local chart \( (\eu{U}, \varphi) \), the local representation 
	\[
	\eu{U} \times (\fs{F})^p \to \fs{E}, \quad (x, v_1, \dots, v_p) \mapsto \omega_{x}(v_1, \dots, v_p)
	\]
	is a Keller's \( C_c^k \)-map.  We denote by \( \Omega^p(\fs{M}, \fs{E}) \) the space of \( \fs{E} \)-valued \( p \)-forms, and by \( \Omega(\fs{M}, \fs{E}) \) the graded space of all \( \fs{E} \)-valued differential forms on \( \fs{M} \), and identify \( \Omega^0(\fs{M}, \fs{E}) \) with the space \( C_c^{\infty}(\fs{M}, \fs{E}) \).
	
	(b) Let \( \fs{E}_1, \fs{E}_2, \fs{E}_3 \) be Fr\'echet spaces and \( \beta\colon \fs{E}_1 \times \fs{E}_2 \to \fs{E}_3 \) be a continuous bilinear map. The {wedge product} 
	\[
	\wedge\colon \Omega^p(\fs{M}, \fs{E}_1) \times \Omega^q(\fs{M}, \fs{E}_2) \to \Omega^{p+q}(\fs{M}, \fs{E}_3)
	\]
	is defined point-wise  by
	\[
	(\omega \wedge \eta)_m(v_1, \dots, v_{p+q}) \coloneqq \frac{1}{p!q!} \sum_{\sigma \in S_{p+q}} \operatorname{sgn}(\sigma) \beta\bigl(\omega_m(v_{\sigma(1)}, \dots, v_{\sigma(p)}), \eta_m(v_{\sigma(p+1)}, \dots, v_{\sigma(p+q)})\bigr).
	\]
	%	The direct sum 
	%	\[ \Omega(\fs{M}, \fs{E}) \coloneqq \bigoplus_{p \ge 0} \Omega^p(\fs{M}, \fs{E}) \]
	%	forms a graded \(\Omega(\fs{M}, \mathbb{R})\)-module under the wedge product.
	
	(c) The definition of the {exterior differential} \( \mathrm{d} \colon \Omega^p(\fs{M}, \fs{E}) \to \Omega^{p+1}(\fs{M}, \fs{E}) \) is uniquely determined by property that for any open subset \( \eu{U} \subseteq \fs{M} \) and vector fields  \( X_0, \dots, X_p \in \eu{V}_{C_c}^{k-1}(\eu{U}) \), the \( \fs{E} \)-valued function \( (\mathrm{d}\omega)(X_0, \dots, X_p) \in C_c^{k-1}(\eu{U}, \fs{E}) \) is given by 
	\[
	\begin{aligned}
		(\mathrm{d}\omega)(X_0, \dots, X_p) \coloneqq & \sum_{i=0}^p (-1)^i X_i \bigl( \omega(X_0, \dots, \widehat{X}_i, \dots, X_p) \bigr) \\
		& + \sum_{i<j} (-1)^{i+j} \omega([X_i, X_j], X_0, \dots, \widehat{X}_i, \dots, \widehat{X}_j, \dots, X_p).
	\end{aligned}
	\]
	Here, \( X_i(f)  \colon \fs{M} \to \fs{E}\) is given by \(X_i(f) \coloneqq \mathrm{D}f \circ X_i\), For a function \( f \in C_c^{k}(\fs{M}, \fs{E}) \) and a vector field \( X \), we set \( X(f)(x) \coloneqq \mathrm{D}f(x)(X(x)) \in \fs{E} \).
	
	(d) For any vector field \( Y \in \eu{V}_{C_c}^{k-1}(\fs{M}) \), the Lie derivative along \( Y \) is a unique linear map given by
	\begin{gather*}
		\mathcal{L}_Y \colon \Omega^p(\fs{M}, \fs{E}) \to \Omega^{p}(\fs{M}, \fs{E}) \quad (\text{for } k \geq 2)\\
		(\mathcal{L}_Y \omega)(X_1, \dots, X_p) \coloneqq Y \bigl( \omega(X_1, \dots, X_p) \bigr) - \sum_{j=1}^p \omega(X_1, \dots, [Y, X_j], \dots, X_p).
	\end{gather*}
	Here \( \omega \in \Omega^p(\fs{M}, \fs{E}) \) and vector fields \( X_1, \dots, X_p \in \eu{V}_{C_c}^{k-1}(\eu{U}) \). Furthermore, for each vector field \( X \in \eu{V}_{C_c}^{k-1}(\fs{M}) \) and \( p \geq 1 \), there exists a unique linear map 
	\begin{gather*}
		i_X \colon \Omega^p(\fs{M}, \fs{E}) \to \Omega^{p-1}(\fs{M}, \fs{E})\\
		(i_X \omega)_x(v_1, \dots, v_{p-1}) \coloneqq \omega_x(X(x), v_1, \dots, v_{p-1}),
	\end{gather*}
	where \( v_1, \dots, v_{p-1} \in \mathrm{T}_x\fs{M} \). For a \( 0 \)-form \( \omega \in \Omega^0(\fs{M}, \fs{E}) = C_c^k(\fs{M}, \fs{E}) \), we define \( i_X \omega \coloneqq 0 \).
\end{definition}
For an \( \fs{E} \)-valued \( p \)-form \( \omega \in \Omega^p(\fs{M}, \fs{E}) \) and vector fields \( X_0, \dots, X_p \in \eu{V}_{C_c}^{k-1}(\fs{M}) \), the {exterior derivative} is given by
\begin{align}\label{eq:ex}
	(d\omega)(X_0, \dots, X_p) \coloneqq & \sum_{i=0}^p (-1)^i (\mathcal{L}_{X_i} \omega)(X_0, \dots, \widehat{X}_i, \dots, X_p) \\ \nonumber
	& + \sum_{i<j} (-1)^{i+j+1} \omega([X_i, X_j], X_0, \dots, \widehat{X}_i, \dots, \widehat{X}_j, \dots, X_p).
\end{align}
The proof of this identity is standard and proceeds by {induction on \( p \)} and the application of {Cartan's Magic Formula} \( 	\mathcal{L}_X \omega = (i_X \circ \mathrm{d} + \mathrm{d} \circ i_X) \omega\).

Next we reformulate the Frobenius theorem in terms of differential ideals.

\noindent		
Let \( \omega \in \Omega^2(\fs{M}, \fs{E}) \) and assume that \( E_\omega = \{ v \in \mathrm{T}\fs{M} \mid i_v \omega = 0 \} \) is a subbundle of \( \mathrm{T}\fs{M} \). If \( X, Y \in \eu{V}_{C_c}^{k-1}(\fs{M}, E_\omega) \) are two sections of \( E_\omega \), then
\[
i_{[X,Y]} \omega = \mathcal{L}_X i_Y \omega - i_Y \mathcal{L}_X \omega = -i_Y (\mathrm{d} i_X \omega + i_X \mathrm{d}\omega) = i_X i_Y \mathrm{d}\omega.
\]
This identity shows that if \( \mathrm{d}\omega = 0 \), then \( X, Y \in \eu{V}_{C_c}^{k-1}(\fs{M}, E_\omega) \) implies that \( [X, Y] \in \eu{V}_{C_c}^{k-2}(\fs{M}, E_\omega) \), and so \( E_\omega \) is involutive. 
For any subbundle \( E \subseteq \mathrm{T}\fs{M} \), the \( p \)-annihilator of \( E \), denoted \( E^0(p) \), is defined as the subspace of global \( p \)-forms that vanish when evaluated on vectors in \( E \):
\[
E^0(p) = \left\{ \alpha \in \Omega^p(\fs{M}, \fs{E}) \;\middle|\; 
\alpha(m)(v_1, \dots, v_p) = \zero{\fs{E}}, \;\; \forall v_1, \dots, v_p \in E_m, \, \forall m \in \fs{M} \right\}.
\]
For each \( p \ge 0 \), \( E^0(p) \) is a linear subspace of \( \Omega^p(\fs{M}, \fs{E}) \) consisting of those forms that vanish when evaluated on smooth vector fields in \( \Vfields{\infty}{\fs{M}}{E} \). The direct sum
\[
I(E) = \bigoplus_{p \ge 0} E^0(p)
\]
is an \textit{ideal} of \( \Omega(\fs{M}, \fs{E}) \) in the sense that it is a graded \( \Omega(\fs{M}, \mathbb{R}) \)-submodule: if \( \omega_1, \omega_2 \in I(E) \) and \( \rho \in \Omega(\fs{M}, \mathbb{R}) \), then \( \omega_1 + \omega_2 \in I(E) \) and \( \rho \wedge \omega_1 \in I(E) \).

For an open subset \( U \subseteq \fs{M} \), we define the subspace of \( \fs{E} \)-valued \( p \)-forms annihilating the subbundle \( E \) by
\[
E^0(p)(U) \coloneqq \left\{ \alpha \in \Omega^p(U, \fs{E}) \mid \alpha(X_1, \dots, X_p) = 0, \quad \forall X_1, \dots, X_p \in \Vfields{\infty}{U}{E} \right\}.
\]
\begin{proposition} \label{prop:frobenius_ideal}
	A subbundle \( E \subseteq \mathrm{T}\fs{M} \) is involutive if for every open subset \( U \subseteq \fs{M} \), the exterior derivative maps the local annihilator of \( E \) into the next degree, i.e.,
	\[
	\mathrm{d}\left( E^0(1)(U) \right) \subseteq E^0(2)(U).
	\]
	Furthermore, if \( E \) is involutive, then \( \omega \in E^0(p)(U) \) implies \( \mathrm{d}\omega \in E^0(p+1)(U) \), so that \( \mathrm{d} I(E|_U) \subseteq I(E|_U) \).
\end{proposition}

\begin{proof}
	Assume that \( \mathrm{d}\left( E^0(1)(U) \right) \subseteq E^0(2)(U) \) for an open set \( U \subseteq \fs{M} \). Let \( X, Y \in \Vfields{\infty}{U}{E} \) and \( \alpha \in E^0(1)(U) \). By \eqref{eq:ex}, we have
	\[
	\mathrm{d}\alpha(X, Y) = \mathcal{L}_X(\alpha(Y)) - \mathcal{L}_Y(\alpha(X)) - \alpha([X, Y]).
	\]
	Since \( \alpha \) annihilates \( E \), \( \alpha(X) = 0 \) and \( \alpha(Y) = 0 \) identically on \( U \). Thus, \( \mathrm{d}\alpha(X, Y) = -\alpha([X, Y]) \). Our assumption implies \( \mathrm{d}\alpha(X, Y) = 0 \), yielding \( \alpha([X, Y]) = 0 \) for all \( \alpha \in E^0(1)(U) \). Assuming the annihilator forms separate points (which holds for closed subbundles in Fr\'echet spaces), this yields \( [X, Y] \in \Vfields{\infty}{U}{E} \). Hence, \( E \) is involutive.
	
	Now, assume \( E \) is involutive and let \( \omega \in E^0(p)(U) \). Then
	\begin{align*}
		\mathrm{d}\omega(X_0, \dots, X_p) &= \sum_{i=0}^p (-1)^i \mathcal{L}_{X_i} \big( \omega(X_0, \dots, \widehat{X_i}, \dots, X_p) \big) \\
		&\quad + \sum_{0 \le i < j \le p} (-1)^{i+j} \omega([X_i, X_j], X_0, \dots, \widehat{X_i}, \dots, \widehat{X_j}, \dots, X_p).
	\end{align*}
	In the first sum, each evaluation of \( \omega \) has \( p \) arguments belonging to \( E \), so it vanishes because \( \omega \in E^0(p)(U) \). In the second sum, since \( E \) is involutive, the bracket \( [X_i, X_j] \in \Vfields{\infty}{U}{E} \). Thus, every evaluation of \( \omega \) again involves \( p \) vector fields in \( E \), causing the entire sum to vanish. Therefore, \( \mathrm{d}\omega(X_0, \dots, X_p) = 0 \), which means \( \mathrm{d}\omega \in E^0(p+1)(U) \).
\end{proof}

\begin{corollary}\label{cor:frobenious}
	Let \( \TB{D} \subseteq \mathrm{T}\fs{M} \) be a split tangent subbundle of a Keller's \( C_c^\infty \)-Fr\'echet manifold \( \fs{M} \) modeled on \( \fs{F} = \fs{F}_1 \oplus \fs{F}_2 \). Assume that for every open subset \( U \subseteq \fs{M} \), the exterior derivative maps the local annihilator of \( \TB{D} \) into the next degree$\colon$
	\[
	\mathrm{d}\left( \TB{D}^0(1)(U) \right) \subseteq \TB{D}^0(2)(U).
	\]
	If \( \TB{D} \) satisfies the local well-posedness condition \( W \) in Definition \ref{def:condition_W}, then \( \TB{D} \) is integrable. 
\end{corollary}

\section*{Acknowledgments}
The author is grateful to the reviewer for pointing out the mistake in the proof of Theorem 3.1, and for providing  insights that helped clarify the proof of Proposition 3.4.

\end{document}